\documentclass[11pt, twoside]{amsart}
\usepackage[english]{babel}
\theoremstyle{plain}

\usepackage{amsmath,amsfonts,amsthm,amssymb}
\usepackage{graphicx}
\usepackage{verbatim}
\usepackage{color}
\usepackage{amscd}
\usepackage{verbatim}
\usepackage{mathrsfs}
\usepackage[colorlinks,citecolor=red,pagebackref,hypertexnames=false, breaklinks]{hyperref}

\addtocounter{page}{1}%
\newtheorem{theorem}{Theorem}[section]

\newtheorem{Thm}[theorem]{Theorem}
\newtheorem{Lem}[theorem]{Lemma}
\newtheorem{Pro}[theorem]{Proposition}

\newtheorem{Rem}[theorem]{Remark}

\newtheorem{Cor}[theorem]{Corollary}
\newtheorem{Conj}[theorem]{Conjecture}

\newtheorem*{corollary*}{Corollary}
\newtheorem*{theorem*}{Theorem}

\def \R {{\mathbb {R}}}

\begin{document}

\setcounter{page}{1}

\newcommand{\M}{{\mathcal M}}
\newcommand{\loc}{{\mathrm{loc}}}
\newcommand{\core}{C_0^{\infty}(\Omega)}
\newcommand{\sob}{W^{1,p}(\Omega)}
\newcommand{\sobloc}{W^{1,p}_{\mathrm{loc}}(\Omega)}
\newcommand{\merhav}{{\mathcal D}^{1,p}}
\newcommand{\be}{\begin{equation}}
\newcommand{\ee}{\end{equation}}
\newcommand{\mysection}[1]{\section{#1}\setcounter{equation}{0}}
\newcommand{\laplace}{\Delta}
\newcommand{\pl}{\laplace_p}
\newcommand{\grad}{\nabla}
\newcommand{\pd}{\partial}
\newcommand{\bo}{\pd}
\newcommand{\csub}{\subset \subset}
\newcommand{\sm}{\setminus}
\newcommand{\ssm}{:}
\newcommand{\diver}{\mathrm{div}\,}
\newcommand{\bea}{\begin{eqnarray}}
\newcommand{\eea}{\end{eqnarray}}
\newcommand{\bean}{\begin{eqnarray*}}
\newcommand{\eean}{\end{eqnarray*}}
\newcommand{\thkl}{\rule[-.5mm]{.3mm}{3mm}}
\newcommand{\cw}{\stackrel{\rightharpoonup}{\rightharpoonup}}
\newcommand{\id}{\operatorname{id}}
\newcommand{\supp}{\operatorname{supp}}
\newcommand{\wlim}{\mbox{ w-lim }}
\newcommand{\mymu}{{x_N^{-p_*}}}
\newcommand{\abs}[1]{\lvert#1\rvert}
\newcommand{\pf}{\noindent \mbox{{\bf Proof}: }}


\renewcommand{\theequation}{\thesection.\arabic{equation}}
\catcode`@=11 \@addtoreset{equation}{section} \catcode`@=12
\newcommand{\Real}{\mathbb{R}}

\newcommand{\B}{\mathbb{B}}

\newcommand{\real}{\mathbb{R}}
\newcommand{\Nat}{\mathbb{N}}
\newcommand{\ZZ}{\mathbb{Z}}
\newcommand{\CC}{\mathbb{C}}
\newcommand{\Pess}{\opname{Pess}}
\newcommand{\Proof}{\mbox{\noindent {\bf Proof} \hspace{2mm}}}
\newcommand{\mbinom}[2]{\left (\!\!{\renewcommand{\arraystretch}{0.5}
\mbox{$\begin{array}[c]{c}  #1\\ #2  \end{array}$}}\!\! \right )}
\newcommand{\brang}[1]{\langle #1 \rangle}
\newcommand{\vstrut}[1]{\rule{0mm}{#1mm}}
\newcommand{\rec}[1]{\frac{1}{#1}}
\newcommand{\set}[1]{\{#1\}}
\newcommand{\dist}[2]{$\mbox{\rm dist}\,(#1,#2)$}
\newcommand{\opname}[1]{\mbox{\rm #1}\,}
\newcommand{\mb}[1]{\;\mbox{ #1 }\;}
\newcommand{\undersym}[2]
 {{\renewcommand{\arraystretch}{0.5}  \mbox{$\begin{array}[t]{c}
 #1\\ #2  \end{array}$}}}
\newlength{\wex}  \newlength{\hex}
\newcommand{\understack}[3]{%
 \settowidth{\wex}{\mbox{$#3$}} \settoheight{\hex}{\mbox{$#1$}}
 \hspace{\wex}  \raisebox{-1.2\hex}{\makebox[-\wex][c]{$#2$}}
 \makebox[\wex][c]{$#1$}   }%
\newcommand{\smit}[1]{\mbox{\small \it #1}}
\newcommand{\lgit}[1]{\mbox{\large \it #1}}
\newcommand{\scts}[1]{\scriptstyle #1}
\newcommand{\scss}[1]{\scriptscriptstyle #1}
\newcommand{\txts}[1]{\textstyle #1}
\newcommand{\dsps}[1]{\displaystyle #1}
\newcommand{\dx}{\,\mathrm{d}x}
\newcommand{\dy}{\,\mathrm{d}y}
\newcommand{\dz}{\,\mathrm{d}z}
\newcommand{\dt}{\,\mathrm{d}t}
\newcommand{\dr}{\,\mathrm{d}r}
\newcommand{\du}{\,\mathrm{d}u}
\newcommand{\dv}{\,\mathrm{d}v}
\newcommand{\dV}{\,\mathrm{d}V}
\newcommand{\ds}{\,\mathrm{d}s}
\newcommand{\dS}{\,\mathrm{d}S}
\newcommand{\dk}{\,\mathrm{d}k}

\newcommand{\dphi}{\,\mathrm{d}\phi}
\newcommand{\dtau}{\,\mathrm{d}\tau}
\newcommand{\dxi}{\,\mathrm{d}\xi}
\newcommand{\deta}{\,\mathrm{d}\eta}
\newcommand{\dsigma}{\,\mathrm{d}\sigma}
\newcommand{\dtheta}{\,\mathrm{d}\theta}
\newcommand{\dnu}{\,\mathrm{d}\nu}

\newcommand{\1}{\mathbf{1}}

\def\ga{\alpha}     \def\gb{\beta}       \def\gg{\gamma}
\def\gc{\chi}       \def\gd{\delta}      \def\ge{\epsilon}
\def\gth{\theta}                         \def\vge{\varepsilon}
\def\gf{\phi}       \def\vgf{\varphi}    \def\gh{\eta}
\def\gi{\iota}      \def\gk{\kappa}      \def\gl{\lambda}
\def\gm{\mu}        \def\gn{\nu}         \def\gp{\pi}
\def\vgp{\varpi}    \def\gr{\rho}        \def\vgr{\varrho}
\def\gs{\sigma}     \def\vgs{\varsigma}  \def\gt{\tau}
\def\gu{\upsilon}   \def\gv{\vartheta}   \def\gw{\omega}
\def\gx{\xi}        \def\gy{\psi}        \def\gz{\zeta}
\def\Gg{\Gamma}     \def\Gd{\Delta}      \def\Gf{\Phi}
\def\Gth{\Theta}
\def\Gl{\Lambda}    \def\Gs{\Sigma}      \def\Gp{\Pi}
\def\Gw{\Omega}     \def\Gx{\Xi}         \def\Gy{\Psi}

\newcommand{\Rn}{\mathbb{R}^n}
 \newcommand{\lanbox}{{\, \vrule height 0.25cm width 0.25cm depth 0.01cm \,}}
  \renewcommand{\qedsymbol}{\lanbox}
   \newcommand{\dd}{\mathrm{d}}
   \newcommand{\Bb}{\mathbb{B}}

\renewcommand{\div}{\mathrm{div}}
\newcommand{\red}[1]{{\color{red} #1}}

\newcommand{\cqfd}{\begin{flushright}                  
			 $\Box$
                 \end{flushright}}

\pagestyle{headings}
\title{Index of the critical catenoid}

\author{Baptiste Devyver}
\address{Baptiste Devyver, Department of Mathematics, Technion, 32000 Haifa, Israel}
\email{devyver@technion.ac.il}

\begin{abstract}

We show that the critical catenoid, as a free boundary minimal surface of the unit ball in $\mathbb{R}^3$, has index $4$. We also prove that a free boundary minimal surface of the unit ball in $\mathbb{R}^3$, that is not a flat disk, has index at least $4$.

\end{abstract}

\maketitle

\tableofcontents

\section{Introduction}

The series of recent articles \cite{FS1}, \cite{FS2} by A. Fraser and R. Schoen triggered a renewed interest in free boundary minimal surfaces. More specifically, Fraser and Schoen studied mainly free boundary minimal surfaces in the unit ball $\Bb^n$ of $\R^n$, and in \cite{FS1}, they provide two interesting examples of those, namely the {\em critical catenoid} (a rescaled piece of the usual catenoid in $\R^3$), and the {\em critical M\"obius band} (a non-orientable free boundary minimal surface in $\Bb^4$). They also constructed in \cite{FS2} free boundary minimal surfaces in $\Bb^3$ of genus $0$ and with an arbitrary number of boundary components. Guided by an hypothetical analogy between (closed) minimal surfaces of $S^3$ and free boundary minimal surfaces of $\Bb^3$, P. Sargent \cite{Sar} and independently L. Ambrozio, A. Carlotto and B. Sharp \cite{ACS2}, have in particular given lower estimates for the Morse index of free boundary minimal surfaces in $\Bb^3$, in terms of their topology, similar to the one found by A. Savo \cite{Sav} for minimal surfaces of the sphere $S^3$. Actually, Savo's result also holds in higher dimensions, and the results \cite{Sar}, \cite{ACS2} are much more general than this since they deal with free boundary minimal hypersurfaces in open sets of $\R^n$ satisfying some convexity assumption (see \cite[Theorem 10]{ACS2} for the strongest result so far obtained). Let us also mention the work \cite{ACS1}, where Savo's results are in particular extended to closed minimal hypersurfaces inside rank one symmetric spaces. A heuristic that has been in the air for some time concerning the analogy between closed minimal surfaces of $S^3$ and free boundary minimal surfaces of $\Bb^3$ is the following: since the Clifford torus is arguably the simplest non-trivial closed minimal surface in $S^3$, and the critical catenoid is the simplest non-trivial free boundary minimal surface in $\Bb^3$, is it possible that the two be ``analogous'' in some sense? There are several celebrated characterizations of the Clifford torus: it is the unique minimal torus in $S^3$ (Lawson's conjecture, proved by S. Brendle \cite{B}), it has minimal Willmore energy among all tori in $S^3$ (Willmore conjecture, proved by F. Coda Marques and A. Neves \cite{CN}), and finally, it has minimal index (equal to $5$) among all non-totally geodesic minimal surfaces in $S^3$ (F. Urbano, \cite{U}). Actually, Urbano's result has been important in Coda Marques and Neves' approach. It is very natural to ask whether a similar characterization to Urbano's for the Clifford torus holds as well for the critical catenoid. However, while it is quite easy to compute the index of the Clifford torus, computing the index of the critical catenoid is harder, due to the presence of the boundary and the more complicated equations. Our main goal in this article is to fill this gap and to show that minimal disks excluded, the critical catenoid has the smallest index possible among all orientable free boundary surfaces of the ball. More precisely, we show the following result:

\begin{Thm}\label{main}

Every orientable, free boundary minimal surface in the unit ball $\Bb^3$, is either a flat disk passing through the origin (and in this case has index $1$), or has index at least $4$. Moreover, the index of the critical catenoid is exactly $4$, so that the latter lower bound for the index is attained.

\end{Thm}

\begin{Rem}
{\em 
The {\em nullity} of the critical catenoid can also be computed. According to \cite[Prop. 6.15]{MNS}, it is equal to $2$.
}
\end{Rem}
In the case of minimal surfaces of $S^3$, according to F. Urbano \cite{U}, the Clifford torus is the unique minimal surface of index $5$. We conjecture that this property also holds for the critical catenoid:

\begin{Conj}\label{uniqueness}

Let $\Sigma$ be a free boundary minimal surface in the unit ball $\Bb^3$ with index $4$. Then, $\Sigma$ is isometric to the critical catenoid. 

\end{Conj}

\begin{Rem}
{\em
While this article was in preparation, the author has been informed of related works by G. Smith and D. Zhou \cite{SZ} on the one hand, and by H. Tran \cite{T} on the other, in which the index of the critical catenoid is also computed. Our proof and H. Tran's use the Steklov spectrum of the Jacobi operator, while G. Smith and D. Zhou's use the Robin spectrum. However all three proofs are different, and have been obtained independently.
}
\end{Rem}
The plan of this article is as follows: in Section 2, we set up the setting and recall well-known definitions. In Section 3, we present the critical catenoid and we introduce suitable coordinates. In Section 4, we study the Dirichlet problem for the stability operator on the catenoid, that is instrumental in our proof of Theorem \ref{main}. In Section 5, we prove the part of Theorem \ref{main} pertaining to the critical catenoid (see Theorem \ref{CC}). In Section 6, we finish the proof of Theorem \ref{main}, by proving the universal lower bound for the index (see Proposition \ref{lower_index}). In Section 7, we present some partial results related to Conjecture \ref{uniqueness}.

\bigskip

\begin{center}{\bf Acknowledgments} \end{center}
The author is thankful to A. Fraser for introducing him to this problem, as well as for many interesting discussions. What is more, all the results contained in Sections 6 and 7 of this paper have been obtained in collaboration with A. Fraser. 

During the time this research was carried out, the author was partially supported by the Natural Sciences and Engineering Research Council of Canada through a post-doctoral fellowship.

\section{Free boundary minimal hypersurfaces in the ball}

Let us first recall the concept of a free boundary minimal hypersurface of the unit ball $\mathbb{B}^{n+1}=\{x\in\R^{n+1}\,;\,||x||\leq 1\}$ of $\R^{n+1}$. For the purpose of this article, a hypersurface (with boundary) of $\R^{n+1}$ is called admissible if it is smoothly embedded in $\Bb^{n+1}$, and if its boundary lies in the unit sphere $S^{n}=\partial\Bb^{n+1}$. A free boundary minimal hypersurface in $\Bb^{n+1}$ is by definition a critical point of the $n$-volume functional, restricted to admissible hypersurfaces. Equivalently, it is an admissible hypersurface $\Sigma^{n}$ which is minimal, and which intersects orthogonally the unit sphere $S^{n}$ at its boundary. This last property is easily seen to be equivalent to the fact that at the boundary of $\Sigma^{n}$, the exterior conormal $\nu$ coincides with the position vector $X$. In this article, we shall be concerned with the {\em index} of these free boundary minimal hypersurfaces. We will assume that $\Sigma^{n}$ is oriented, and denote by $N$ a smooth unit normal. The index of a free boundary minimal hypersurface in $\Bb^{n+1}$ can be defined as the (Morse) index of the quadratic form associated with the second variation of the $n$-volume functional, defined on the set of admissible hypersurfaces. It is well-known that if $\Sigma$ is a free boundary minimal hypersurface, and $\Sigma_t$, $t\in [0,1)$ is a one-parameter family of admissible, normal deformations of $\Sigma$ (thus, $\Sigma_0=\Sigma$), then 

$$\frac{\partial}{\partial t}\Big|_{t=0} \mathrm{Vol}(\Sigma_t)=0,\,\,\,\frac{\partial^2}{\partial t^2}\Big|_{t=0} \mathrm{Vol}(\Sigma_t)=\int_\Sigma(|\nabla u|^2-|A|^2u^2)-\int_{\partial\Sigma}u^2.$$
Here, $u$ is a smooth, real function on $\Sigma$, such that the deformation $\Sigma_t$ occurs at $t=0$ in the direction of the normal vector field $uN$. Also, $A$ denotes the second fundamental form, and $|A|^2$ is the square of its norm, i.e. the sum of the squares of the principal curvatures. This leads one to consider the quadratic form

$$Q(u)=\int_\Sigma(|\nabla u|^2-|A|^2u^2)-\int_{\partial\Sigma}u^2.$$
The quadratic form $Q$ is naturally associated to a second-order differential operator, the {\em Jacobi operator} $J=\Delta-|A|^2$ (in this article, we shall take the convention that $\Delta$ has non-negative spectrum). The (Morse) index of $\Sigma$, as a free boundary minimal hypersurface, is then defined as the maximal dimension of a vector space of smooth functions on $\Sigma$, on which $Q$ is negative definite. The index has been extensively studied for closed hypersurfaces, in which case it is simply equal to the number of negative eigenvalues of the Jacobi operator. However, the presence of a boundary term in $Q$ makes things more difficult. It is well-known (see \cite{CFP}) that the index of $\Sigma$ is equal to the number of negative eigenvalues of the following Robin eigenvalue problem:

$$\left\{\begin{array}{lcl}
Ju=\lambda u\mbox{ on }\Sigma\\\\
\frac{\partial u}{\partial\nu}= u \mbox{ on }\partial\Sigma,
\end{array}\right.$$
However, in the free boundary case several other legitimate eigenvalue problems can be considered for $J$, in connection to the index (for example, Dirichlet, Steklov, etc...); and indeed, in order to compute the index of the critical catenoid, we will study the Steklov problem, and not the Robin one.

To conclude this preliminary section, let us recall an interesting result, proved by Fraser and Schoen in \cite{FS2}. For $v$ a constant vector in $\R^{n+1}$, denote $v^\perp:=(N,v)$, which is a smooth function on $\Sigma$. Note also that $Jv^\perp=0$ for all $v\in\mathbb{R}^3$. That is, $v^\perp$ is a Jacobi field. Then, by \cite[Theorem 3.1]{FS2},

$$Q(v^\perp)=-2\int_\Sigma |v^\perp|^2.$$
If $n=2$, the vector space $\{v^\perp \,;\,v\in \R^{n+1}\}$ has dimension $3$, unless $\Sigma$ is a flat disk. Thus, in particular, every free boundary minimal surface of $\Bb^3$, that is not a flat disk, has index at least $3$. Our main result in this article implies that this lower bound can be improved to $4$, and moreover $4$ is optimal. Of course, a flat disk is a free boundary minimal surface of the unit ball if and only if it passes through the origin, and in this case it is not hard to check that it has index $1$, by using polar coordinates and Fourier analysis. 

\section{The critical catenoid}

In this section, we present the main protagonist of this article, namely the critical catenoid, and give some useful properties related to it. Let $\mathcal{C}_a$, $a>0$, be the catenoid in $\R^3$ parametrized by

$$X_a(t,\theta)=(a\cosh(t/a)\cos(\theta),a\cosh(t/a)\sin(\theta),t)=a(\cosh(s)\cos(\theta),\cosh(s)\sin(\theta),s),\,s=t/a.$$
On $\mathcal{C}_a$, the unit normal is given by

$$N=\left(-\frac{\cos(\theta)}{\cosh(t/a)},-\frac{\sin(\theta)}{\cosh(t/a)},\tanh(t/a)\right).$$
A simple computation shows that there exists a unique value of the parameter $a$, such that $\mathcal{C}_a$ intersects the unit sphere orthogonally. For this value of $a$, the part of $\mathcal{C}_a$ lying inside the unit ball is called the {\em critical catenoid}. In the above coordinates, it means that $t\in [-aT,aT]$, where $T$ is the unique positive solution of the equation

$$T\tanh(T)=1.$$
With this definition of $T$, the dilation parameter $a$ has the value

$$a=\frac{1}{T\cosh(T)}.$$
\bigskip

\noindent {\bf{\em From now on and until Section 6, $\Sigma$ will denote the critical catenoid as described above. }}

\subsection{A choice of coordinates}

In the coordinates $(s,\theta)$, $s=t/a$, the metric on $\Sigma$ writes

$$g=a^2\cosh^2(s)(ds^2+d\theta^2),$$
with $s\in [-T,T]$ and $\theta\in [0,2\pi)$. The Jacobi operator $J$ is given by

$$Ju(s,\theta)=-\frac{1}{a^2\cosh^2(s)}\left(\frac{\partial^2 u}{\partial s^2}+\frac{\partial^2u}{\partial \theta^2}\right)-\frac{2}{a^2\cosh^4(s)}u,$$
and the quadratic form $Q$ writes

$$\begin{array}{rcl}
Q(u(s,\theta))&=&\int_{[-T,T]\times S^1}\left\{\left|\frac{\partial u}{\partial s}\right|^2+\left|\frac{\partial u}{\partial \theta}\right|^2-\frac{2}{\cosh^2(s)}u^2\right\}\,dsd\theta\\\\
&&-\frac{1}{T}\int_0^{2\pi}\left\{u^2(T,\theta)+u^2(-T,\theta)\right\}\,d\theta
\end{array}$$
It will be convenient to conformally change the metric $g$ on $\Sigma$, in order to see the resulting surface as a part of the unit sphere $S^2$ and hence use spherical coordinates. While not strictly necessary (our proof would work without it), it does simplify the Jacobi operator, and the equation for Jacobi fields becomes easier. Thus, let us consider the metric $\tilde{g}$, conformal to $g$, defined by

$$\tilde{g}=(\cosh(s))^{-2}(ds^2+d\theta^2).$$
Note that $\tilde{g}=K_g\,g$, where $K_g$ is the Gauss curvature of $g$. and it follows that the metric $\tilde{g}$ has Gauss curvature constant and equal to $1$: indeed, since $\Sigma$ is minimal, the metric $\tilde{g}$ is the pull-back of the canonical metric on the $2$-sphere by the Gauss map, hence has Gauss curvature constant and equal to $1$ (outside of the critical points of the Gauss map). Moreover, it is easily seen that $(\Sigma,\tilde{g})$ is isometric to an annulus inside the $2$-dimensional sphere $S^2$: indeed, let

$$
\varphi=2\arctan(e^{-s}),
$$
then by use of standard trigonometric formulas, one computes that

$$
\cosh^{-1}(s)=\sin(\varphi).
$$
Therefore, the metric $\tilde{g}$ can be written as

$$\tilde{g}=d\varphi^2+\sin^2(\varphi)d\theta^2=g_{(S^2,can)},$$
the canonical metric on $S^2$ written in spherical coordinates $(\varphi,\theta)$. The surface $(\Sigma,\tilde{g})$ is therefore isometric to the annulus $\Omega:=\{\varphi^*\leq \varphi\leq \pi-\varphi^*,\,0\leq \theta<2\pi\}$ of the unit sphere $S^2$. Also, observe that

$$
\tanh(s)=\cos(\varphi).
$$
Therefore, since $-T\leq s\leq T$, and $T\tanh(T)=1$, one obtains

$$(\theta,\varphi)\in\Omega\Leftrightarrow -\frac{1}{T}\leq \cos(\varphi)\leq \frac{1}{T}.$$
Concerning the quadratic form $Q$, using the conformal invariance of the Dirichlet energy in dimension $2$, one can write $Q$ as

$$Q(u)=\int_{\Omega}(|\nabla u|^2-2u^2)-\frac{\cosh(T)}{T}\int_{\partial\Omega}u^2\,d\sigma,$$
where all the quantities (gradient, volume, etc) are computed in the canonical metric of $S^2$. Thus, one sees that the operator naturally associated to $Q$ in the metric $\tilde{g}$ is

$$\tilde{J}=\Delta_{S^2}-2.$$
Hence, as promised, in the conformal metric $\tilde{g}=g_{(S^2,can)}$ the operator associated to $Q$ has a particularly simple expression.

\subsection{Jacobi fields}

Recall that, by a slight abuse of notation, a smooth function $u$ on $\Sigma$ is called a Jacobi field if $Ju=0$. For the critical catenoid, there are several interesting Jacobi fields that one can consider: one example that we have already mentioned in Section 2 is $v^\perp=(v,N)$ for $v$ a constant vector in $\R^3$. Interestingly, if one considers the canonical basis $(v_x,v_y,v_z)$ of $\R^3$, one obtains three Jacobi fields $v_x^\perp$, $v_y^\perp$ and $v_z^\perp$, that have the extra property of being Steklov eigenfunctions for $J$; that is, they solve a Steklov boundary value problem:

\begin{equation}\label{Steklov}
\left\{\begin{array}{lll}

Ju=0,&\mbox{ on }\Sigma,\\
\frac{\partial u}{\partial \nu}=\lambda u,&\mbox{ on }\partial\Sigma,
\end{array}\right.
\end{equation}
where $\nu$ is the exterior unit normal to $\Sigma$ on $\partial \Sigma$ (which is also equal to $X$, the position vector, by the free boundary condition). More precisely, by a straightforward computation, one checks:

\begin{Lem}

The function $v_x^\perp$, $v_y^\perp$ and $v_z^\perp$ are Steklov eigenfunctions for the Jacobi operator $J$, associated respectively to the eigenvalues $-1$, $-1$ and $\frac{1}{\sinh^2(T)}$.

\end{Lem}
Notice that by Green's formula, if $u$ satisfies \eqref{Steklov}, then $Q(u)<0$ if and only if $\lambda<1$. This is the case for the functions $v_x^\perp$, $v_y^\perp$ and $v_z^\perp$. It will also be useful for future reference to write down the boundary values of these functions; to this end, write $\partial\Sigma=\partial\Sigma_+\cup\partial\Sigma_-$, with $\partial\Sigma_\pm:=\{(s,\theta)\,;\,s=\pm T,\,\theta\in [0,2\pi)\}$. Then,

$$v_x^\perp=\frac{1}{\cosh(T)}\cos(\theta),\, v_y^\perp=\frac{1}{\cosh(T)}\sin(\theta),\,v_z^\perp=\frac{\pm 1}{T}.$$
According to \cite[Propositions 2.1, 2.2] {BSE}, there are also three Jacobi fields coming from Killing (rotation) vector fields of $\R^3$, and one Jacobi field coming from varying the parameter $a$ defining the catenoid. The former are $v_{(x,z)}^\perp:=(-z\partial_x+x\partial_z,N)$, $v_{(y,z)}^\perp:=(-z\partial_y+y\partial_z,N)$ and $v_{(x,y)}^\perp:=-y\partial_x+x\partial_y,N)$, while the latter is $\xi:=-\left(\frac{\partial}{\partial a}X_a(t,\theta),N\right)$. By the fact that rotations with center at the origin preserve the unit ball and the definition of the quadratic form $Q$ in terms of second variation of the area functional (see Section 2), it is obvious that 

$$Q(v^\perp_{(x,z)})=Q(v^\perp_{(y,z)})=Q(v^\perp_{(x,y)})=0.$$
Explicitly, in coordinates, one has

$$v^\perp_{(x,z)}=\Lambda(s)\cos(\theta),\,v^\perp_{(y,z)}=\Lambda(s)\sin(\theta),\,v_{(x,y)}^\perp=0,$$
with 

\begin{equation}\label{Lambda}
\Lambda(s)=a\left(\frac{s}{\cosh(s)}+\tanh(s)\cosh(s)\right).
\end{equation}
It is easily checked that $v^\perp_{(x,z)}$ and $v^\perp_{(y,z)}$ are Steklov eigenfunctions for $J$, with associated eigenvalue $1$. Concerning the Jacobi field $\xi$, one computes in coordinates that

$$\xi=1-s\tanh(s).$$
From this computation, and recalling that $s$ ranges from $-T$ to $T$ with $T\tanh(T)=1$, one concludes that $\xi$ is positive in the interior of $\Sigma$ and vanishes on the boundary of $\Sigma$. The non-negativity of $\xi$ implies that it is a first eigenfunction of $J$ with Dirichlet boundary conditions. If we denote by $\lambda_1(J)<\lambda_2(J)\leq \cdots\leq \lambda_n(J)\leq \cdots$ the spectrum of $J$ with Dirichlet boundary conditions, it follows that $\lambda_1(J)=0$ (thus, $\Sigma$ is a maximal stable domain of a rescaled catenoid, for variations preserving the boundary). Of course, this is a well-known fact (see \cite{BSE}). 

\subsection{A Fourier decomposition} We introduce a natural Fourier decomposition for Jacobi fields: working in coordinates $(\varphi,\theta)$ on $\Sigma$, and assuming that $u$ is a smooth, real function on $\Sigma$, one can write $u(\varphi,\cdot)$ as the sum of its Fourier series in $\theta$:

\begin{equation}\label{Fourier_coef}
u(\varphi,\theta)=a_0(\varphi)+\sum_{n=1}^\infty(a_n(\varphi)\cos(n\theta)+b_n(\varphi)\sin(n\theta)).
\end{equation}
Using the well-known formula $\Delta_{\tilde{g}}=\lambda^{-2}\Delta_g$ relating the Laplacian $\Delta_{\tilde{g}}$ of a metric $\tilde{g}=\lambda^2g$ conformal to $g$, one easily finds that

$$\tilde{J}=\lambda^{-2}J.$$
According to our computations in Section 3.1, $\tilde{J}=\Delta_{(S^2,can)}-2$. Hence, $Ju=0$ if and only if in $(\varphi,\theta)$ coordinates, $(\Delta_{S^2}-2)u(\varphi,\theta)=0$. Recalling the expression of the Laplacian of the unit sphere in spherical coordinates,

$$-\Delta=\frac{1}{\sin(\varphi)}\frac{\partial}{\partial \varphi}\left(\sin(\varphi)\frac{\partial}{\partial \varphi}\right)+\frac{1}{\sin^2(\varphi)}\frac{\partial^2}{\partial\theta^2},$$
one finds that $u$ is a Jacobi field if and only if for every $n\geq0$ (resp. $n\geq1$), the functions $a_n(\varphi)$ (resp. $b_n(\varphi)$) are solutions of the following ODE, known in the literature as the associated Legendre equation with indices $(1,n)$ (see for example \cite{AS}):

\begin{equation}\label{Fourier_eq}
\Big\{-\frac{1}{\sin(\varphi)}\frac{\partial}{\partial \varphi}\left(\sin(\varphi)\frac{\partial}{\partial \varphi}\right)+\left(\frac{n^2}{\sin^2(\varphi)}-2\right)\Big\}a_n(\varphi)=0,
\end{equation}
for $\varphi\in [\varphi^*,\pi-\varphi^*]$. This suggests to introduce the differential operators $\mathscr{L}_n$, defined by

\begin{equation}\label{Fourier_op}
\mathscr{L}_n=-\frac{1}{\sin(\varphi)}\frac{\partial}{\partial \varphi}\left(\sin(\varphi)\frac{\partial}{\partial \varphi}\right)+\left(\frac{n^2}{\sin^2(\varphi)}-2\right).
\end{equation}
The operator $\mathscr{L}_0$ actually corresponds to the ``radial part'' of the Jacobi operator $J$: more precisely, if $u$ depends only on $s$ (or, equivalently, on $\varphi$), then $Ju=0$ if and only if $\mathscr{L}_0u=0$. We observe that the equation $\mathscr{L}_0\,a(\varphi)=0$ is a regular ODE on the interval $[\varphi^*,\pi-\varphi^*]$, and thus has a space of dimension $2$ of solutions: explicitly, a basis of the space of solutions is provided by $\xi$ and $v_z^\perp$. Consequently, if $u$ is a Jacobi field, i.e. if $Ju=0$, then its Fourier mode of order zero, $a_0(\varphi)$, is a linear combination of $\xi$ and $v_z^\perp$. 

Concerning the equation $\mathscr{L}_1\,a(\varphi)=0$, which is also a regular ODE, a basis of solution is obtained by looking at the expression of the Jacobi fields $v_x^\perp$ and $v_y^\perp$ on the one end, and $v_{(x,z)}^\perp$ and $v_{(y,z)}^\perp$ on the other hand, in $(\varphi,\theta)$ coordinates. Indeed,

$$v_x^\perp=-\chi(s)\cos(\theta),\,v_y^\perp=-\chi(s)\sin(\theta),$$
with $\chi(s)=\cosh^{-1}(s)=\sin(\varphi)$ satisfying $\mathscr{L}_1\chi=0$. The other independent solution of  $\mathscr{L}_1\,a(\varphi)=0$ is obtained by writing the function $\Lambda(s)$ from \eqref{Lambda} as a function of $\varphi$ rather than $s$; by a slight abuse of notation, we will also denote this function $\Lambda(\varphi)$. Thus, if $u$ is a Jacobi field, then its Fourier mode of order $1$, $a_1(\varphi)\cos(\theta)+b_1(\varphi)\sin(\theta)$, is a linear combination of the four Jacobi fields $v_x^\perp$, $v_y^\perp$, $v_{(x,z)}^\perp$ and $v_{(y,z)}^\perp$. 

Altogether, for a Jacobi field $u$, the Fourier decomposition in the $\theta$ variable can be equivalently written as

\begin{equation}\label{Fourier_detail}
u(\varphi,\theta)=a v_z^\perp+b\xi+\alpha v_x^\perp+\beta v_y^\perp+\gamma v_{(x,z)}^\perp+\delta v_{(y,z)}^\perp+\sum_{n= 2}^\infty(a_n(\varphi)\cos(n\theta)+b_n(\varphi)\sin(n\theta)),
\end{equation}
where $a,b,\alpha,\beta,\gamma,\delta$ are real numbers and $a_n(\varphi)$, $b_n(\varphi)$ are solutions of \eqref{Fourier_eq}.

\section{The Dirichlet problem}

A crucial ingredient in our proof that the critical catenoid has index $4$, is the solution of the Dirichlet problem for the stability operator $J$. However, since $\Sigma$ is a maximal domain of stability, the operator $J$ has a non-trivial kernel. Thus, bearing in mind the Fredholm alternative, we introduce a zero-flux condition as follows. For a smooth function $u$ on $\Sigma$, we define its {\em flux} on the boundary:

$$\Phi(u)=\int_{\partial\Sigma}\frac{\partial u}{\partial \nu}.$$
We will say that $u$ has zero flux at the boundary if $\Phi(u)=0$. 

\begin{Lem}\label{Dirichlet}

Let $u$ be a smooth function on $\partial\Sigma$. Then, the Dirichlet problem

$$\left\{\begin{array}{lcl}
J\hat{u}=0\,\mbox{on }\Sigma\\\\
\hat{u}=u\,\mbox{on }\partial\Sigma,
\end{array}\right.$$
is solvable if and only if $\int_{\partial\Sigma}u=0$, and in this case it has a unique solution $\hat{u}$ with zero flux at the boundary.
\end{Lem}

\begin{proof}

Let $\tilde{u}$ be a smooth extension of $u$ to $\Sigma$ (supported in a tubular neighborhood of $\partial\Sigma$). Then, writing $\hat{u}=v+\tilde{u}$ and $f=-J\tilde{u}$, we see that solving the Dirichlet problem is equivalent to solving

\begin{equation}\label{DBP2}
\left\{\begin{array}{lcl}

Jv=f\,\mbox{on }\Sigma,\\
v=0\mbox{ on }\partial \Sigma.

\end{array}\right.
\end{equation}
By the Fredholm alternative (see \cite[Thm 6.2.4]{E}), there exists a solution of \eqref{DBP2} if and only if 

$$\int_\Sigma fw=0,$$
for every $w$ such that $J^*w=0$, $w|_{\partial\Sigma}=0$. Since $J$ is self-adjoint, $J^*=J$. Also, $0$ being the bottom of the spectrum of $J$ with Dirichlet boundary conditions, it has multiplicity one as a Dirichlet eigenvalue of $J$. Thus, \eqref{DBP2} has a solution if and only if

$$\int_{\Sigma}(J\tilde{u})\xi=0.$$
But by Green's formula,

$$\int_\Sigma (J\tilde{u})\xi-(J\xi)\tilde{u}=\int_{\partial\Sigma}u\frac{\partial\xi}{\partial\nu}-\xi\frac{\partial \tilde{u}}{\partial\nu}.$$
Taking into account that $J\xi=0$, $\xi|_{\partial\Sigma}=0$ and $\tilde{u}|_{\partial\Sigma}=u$, we obtain that \eqref{DBP2} has a solution if and only if

$$0=\int_{\partial\Sigma}u\frac{\partial\xi}{\partial\nu}.$$
Using that $\frac{\partial \xi}{\partial \nu}$ is constant on $\partial \Sigma$, we deduce that $\int_{\partial\Sigma}u=0$ is a necessary and sufficient condition for solving the Dirichlet problem. Concerning uniqueness, just note that if $u_1$ and $u_2$ are both solutions of the same Dirichlet problem, then $u:=u_1-u_2$ satisfies $Ju=0$ and $u|_{\partial\Sigma}=0$, which implies that $u$ is a multiple of $\xi$ (and conversely). Since $\int_{\partial\Sigma}\frac{\partial\xi}{\partial\nu}\neq 0$, it follows that if the Dirichlet problem is solvable then there is precisely one solution with zero flux.

\end{proof}

\section{Index of the critical catenoid}

Our aim in this section is to compute the index of the critical catenoid:

\begin{Thm}\label{CC}

The Morse index of the critical catenoid in the unit ball $\mathbb{B}^3$ is exactly $4$. 

\end{Thm}
The idea of the proof is to show that the vector space $\mathcal{W}$ generated by the constant function $\1$ and the Jacobi fields $v^\perp$, $v\in\R^3$, is a $4$-dimensional space on which $Q$ is negative definite, and with the following property: if $u$ is a smooth function on $\Sigma$ such that for every $w$ in $\mathcal{W}$, $Q(u,w)=0$ (in other words, if $u$ is $Q$-orthogonal to $\mathcal{W}$), then $Q(u)\geq0$. Then, that the index is $4$ follows by a standard projection argument. In order to make this idea work, an important point will be to decompose $u$, thanks to the solution of the Dirichlet problem in Section 4, as the sum of a Jacobi field, and of a smooth function vanishing at the boundary of $\Sigma$.

\subsection{Index is at least $4$}
Later, in Section 5, we will prove that every free boundary minimal surface in the unit ball of $\R^3$, that is not a flat disk, has index at least $4$. However, in the case of the critical catenoid, there is an easy elementary argument, which we present now. Consider the Jacobi fields $v_x^\perp$, $v_y^\perp$ and $v_z^\perp$. Recalling their boundary values (see Section 3.2), it is easily seen that $v_x^\perp$, $v_y^\perp$ and $v_z^\perp$ are pairwise $L^2$-orthogonal when restricted to $\partial\Sigma$. We have already mentioned in Section 3.2 that $v_x^\perp$, $v_y^\perp$ and $v_z^\perp$ are Steklov eigenvalues for the operator $J$, that is they satisfy the boundary value problem \eqref{Steklov}. But by Green's formula, if $Ju=0$ and $\frac{\partial u}{\partial\nu}=\lambda u$, then

$$\begin{array}{rcl}
Q(u,v)&=&\int_{\partial\Sigma}v(\frac{\partial u}{\partial\nu}-u)\\\\
&=& (\lambda-1)\int_{\partial\Sigma}uv.
\end{array}$$
and so we conclude from the pairwise $L^2$-orthogonality of $v_x^\perp$, $v_y^\perp$ and $v_z^\perp$ on $\partial\Sigma$ that

$$0=Q(v_x^\perp,v_y^\perp)=Q(v_x^\perp,v_z^\perp)=Q(v_y^\perp,v_z^\perp).$$
Also, if one considers the constant function $\1$ on $\Sigma$, one has obviously

$$Q(\1)=-\int_\Sigma|A|^2-L(\partial\Sigma)<0,$$
and moreover by the above Green formula, and by the fact that 

$$0=\int_{\partial\Sigma}v_x^\perp=\int_{\partial\Sigma}v_y^\perp=\int_{\partial\Sigma}v_z^\perp,$$
one sees that 

$$Q(\1,v_x^\perp)=Q(\1,v_y^\perp)=Q(\1,v_z^\perp)=0.$$
From this discussion, we conclude that the quadratic form $Q$ is negative definite on the vector space $\mathcal{W}$ generated by $\1$ and the $v^\perp$, $v\in \R^3$. Since this is a $4$-dimensional space, we conclude that the index of $\Sigma$ is at least $4$.

\subsection{Index and Jacobi fields} Before proving the upper bound for the index, we need a crucial preliminary result, which states as follows:

\begin{Pro}\label{positivity}

Let $u$ be a Jacobi field on $\Sigma$, such that

$$0=\int_{\partial\Sigma}u=\int_{\partial\Sigma}uv^\perp,\,v\in \R^3.$$
Then,

$$Q(u)\geq0.$$

\end{Pro}

\begin{proof}

We work in $(\varphi,\theta)$ coordinates and use the Fourier series \eqref{Fourier_detail} in the $\theta$ variable, introduced in Section 3.3:

$$
u(\varphi,\theta)=a v_z^\perp+b\xi+\alpha v_x^\perp+\beta v_y^\perp+\gamma v_{(x,z)}^\perp+\delta v_{(y,z)}^\perp+\sum_{n= 2}^\infty(a_n(\varphi)\cos(n\theta)+b_n(\varphi)\sin(n\theta)),
$$
where $a,b,\alpha,\beta,\gamma,\delta$ are real numbers. We first claim that under the hypotheses of Proposition \ref{positivity}, $a=\alpha=\beta=0$. Indeed, one clearly has

$$0=\int_{\partial \Sigma} uv_z^\perp=a\int_{\partial \Sigma} |v_z^\perp|^2,$$
hence $a=0$. Next we show that $\alpha=0$. Clearly, 

$$\int_{\partial \Sigma} uv_x^\perp=\alpha\int_{\partial \Sigma} |v_x^\perp|^2+\gamma\int_{\partial \Sigma}v_x^\perp v^\perp_{(x,z)},$$
all the other terms obviously vanishing by the well-known orthogonality properties of trigonometric functions. But recall that $v_x^\perp=\chi(\varphi)\cos(\theta)$ with $\chi$ even, whereas $v^\perp_{(x,z)}=\Lambda(\varphi)\cos(\theta)$ with $\Lambda$ odd. Therefore,

$$\int_{\partial \Sigma}v_x^\perp v^\perp_{(x,z)}=0.$$
Consequently,

$$0=\int_{\partial \Sigma} uv_x^\perp=\alpha\int_{\partial\Sigma} |v_x^\perp|^2,$$
and it follows that $\alpha=0$. A similar argument yields $\beta=0$.

Next, it follows easily from Green's formula that the above Fourier decomposition of $u$  is $Q$-orthogonal (to treat the two terms $Q(v_x^\perp,v_{(x,z)}^\perp)$ and $Q(v_y^\perp,v_{(y,z)}^\perp)$, use as above the fact that $\chi$ is even and $\Lambda$ is odd). Hence, using that $0=Q(\xi)=Q(v_{(x,z})^\perp)=Q(v_{(y,z)}^\perp)$, one gets

$$Q(u)=\sum_{n= 2}^\infty Q(a_n(\varphi)\cos(n\theta))+Q(b_n(\varphi)\sin(n\theta)).$$
For $n\geq 2$, define a quadratic form $Q_n$ on smooth function $f(\varphi)$, $\varphi\in[\varphi^*,\pi-\varphi^*]$, by requiring that

$$Q_n(f(\varphi))=Q(f(\varphi)\cos(n\theta))=Q(f(\varphi)\sin(n\theta)).$$
More explicitly, $Q_n$ is given by the following formula:

 $$2Q_n(f(\varphi))=\int_{\varphi^*}^{\pi-\varphi^*}\left(|f^\prime(\varphi)|^2+\left(\frac{n^2}{\sin^2(\varphi)}-2\right)f^2(\varphi)\right)\sin(\varphi)\,d\varphi-\frac{1}{T}(f(\pi-\varphi^*)^2+f(\varphi^*)^2).$$
 Thus,
 
 \begin{equation}\label{orthogonal}
Q(u)=\sum_{n=2}^\infty Q_n(a_n(\varphi))+Q_n(b_n(\varphi)).
 \end{equation}
 It is obvious that if $n\geq2$,
 
 $$Q_n(f(\varphi))\geq Q_2(f(\varphi)).$$
 The result of Proposition \ref{positivity} is a direct consequence of the following claim:
 
 \bigskip
\noindent {\bf Claim}: $$\boxed{\mathrm{The\,\, quadratic \,\,form\,\,} Q_2\mathrm{ \,\,is\,\, positive\,\, definite}.}$$
 
 In order to prove this claim, we need a criterion for positivity of quadratic forms having a boundary term:
 
 \begin{Lem}\label{ground_state}
 
 Let $S$ be a quadratic form acting on functions $f(\varphi)$, $\varphi\in [a,b]$, defined as
 
 $$S(f)=\int_a^b \left((f^\prime)^2+Vf^2\right)m(\varphi)d\varphi-\alpha(f^2(a)+f^2(b)),$$
 where $m$ is a positive, smooth function, $\alpha$ is a real number, and $V$ is a real potential on $[a,b]$. Assume that there is a positive function $h(\varphi)$, $\varphi\in [a,b]$, such that
 
$$\left(\frac{1}{m}\frac{d}{d\varphi}\left(m\frac{d }{d\varphi}\right)-V\right)h\leq0,$$ 
and such that 
 
$$-m(a)\frac{d}{d\varphi}\Big|_{\varphi=a} (\log h)>\alpha ,\,\,m(b)\frac{d}{d\varphi}\Big|_{\varphi=b} (\log h)>\alpha.$$
Then, $S$ is positive definite.

\end{Lem}

\noindent{\em Proof of Lemma \ref{ground_state}:} since $h$ is positive, one can write $f(\varphi)$ in the form

$$f(\varphi)=g(\varphi)h(\varphi).$$
One then expands the term $(f^\prime)^2$, to find that

$$S(f)=\int_a^b \left(g^2(h^\prime)^2+\frac{1}{2}(g^2)^\prime(h^2)^\prime+h^2(g^\prime)^2+Vg^2h^2\right)m-\alpha(f^2(a)+f^2(b)).$$
Integrating by parts the term $(g^2)^\prime(h^2)^\prime$ and using the assumptions on $h$, one finds

$$\begin{array}{rcl}
S(f)&=&\int_a^b \left(-\frac{1}{m}\frac{d}{d\varphi}\left(m\frac{d h}{d\varphi}\right)+Vh\right)hg^2m+\int_a^bh^2(g^\prime)^2m+\left[mg^2 h\frac{d h}{d \varphi}\right]_a^b-\alpha(f^2(a)+f^2(b))\\\\
&\geq& \int_a^bh^2(g^\prime)^2m+f^2(a)\left(-m(a)\frac{d}{d\varphi}\Big|_{\varphi=a} (\log h)-\alpha\right)+f^2(b)\left(m(b)\frac{d}{d\varphi}\Big|_{\varphi=b} (\log h)-\alpha\right),\\\\
>0,
\end{array}$$
and the result of Lemma \ref{ground_state} follows.

\cqfd
In order to prove the claim, one uses Lemma \ref{ground_state}; for that, one needs to find a suitable positive solution $h(\varphi)$ of \eqref{Fourier_eq} for $n=2$. Letting $x=\cos(\varphi)$, the equation \eqref{Fourier_eq} for $n=2$ becomes

\begin{equation}\label{Legendre}
\left\{-\frac{d}{dx}\left((1-x^2)\frac{d}{dx}\right)+\frac{4}{1-x^2}-2\right\}a(x)=0,\,-\frac{1}{T}\leq x\leq \frac{1}{T}.
\end{equation}
It is easily checked that the function $(1-x^2)^{-1}$ is a positive solution to \eqref{Legendre}. Coming back to the variable $\varphi$, one obtains the positive solution

$$h(\varphi)=\frac{1}{\sin^2(\varphi)}(=\cosh^2(s))$$
to \eqref{Fourier_eq} for $n=2$. One now checks that the $h$ satisfies the hypotheses of Lemma \ref{ground_state}.  We compute that 

$$\frac{d \log h}{d \varphi}=-2\,\mathrm{cotan}(\varphi).$$
For $\varphi=\varphi^*$, this is equal to $-2\tanh(T)\cosh(T)=-2\sinh(T)$. Since $T=\frac{\cosh(T)}{\sinh(T)}$, we get that $2\sinh(T)>\frac{\cosh(T)}{T}$, so 

$$-\frac{1}{\cosh(T)}\frac{d}{d\varphi}\Big|_{\varphi=\varphi^*} (\log h)>\frac{1}{T}.$$
One proves similarly that
 
$$\frac{1}{\cosh(T)}\frac{d}{d\varphi}\Big|_{\varphi=\pi-\varphi^*} (\log h)>\frac{1}{T}.$$
Thus, $h$ satisfies the hypotheses of Lemma \ref{ground_state}, and the claim follows. \\

\end{proof}

\subsection{Index is at most $4$} With the result of Proposition \ref{positivity} at hand, one can now finish the proof of Theorem \ref{CC} concerning the index of the critical catenoid. As mentioned before, the idea is to prove that $Q(u)\geq 0$, as soon as $0=Q(u,\1)=Q(u,v_x^\perp)=Q(u,v_y^\perp)=Q(u,v_z^\perp)$. Let us define an auxiliary function

$$f=u+a\1,$$
where $a$ is chosen so that 

$$\int_{\partial \Sigma}f=0.$$
Since $Q(v^\perp,\1)=0$, $v\in\R^3$, one has 

$$Q(f,v^\perp)=0,\,v\in\R^3.$$
Furthermore, since $Q(u,\1)=0$,

$$Q(u)=Q(f)-a^2Q(\1).$$
Since $Q(\mathbf{1})<0$, we see that in order to conclude that $Q(u)\geq0$, it is enough to show that $Q(f)\geq0$. Let $h$ be the unique Jacobi field with zero flux at the boundary, such that $h|_{\partial\Sigma}=f|_{\partial\Sigma}$ (the existence of $h$ is guaranteed by Lemma \ref{Dirichlet}). Then,

$$f=h+g,$$
with $g|_{\partial\Sigma}\equiv0$. By Green's formula,

$$Q(h,g)=\int_\Sigma (Jh)g+\int_{\partial\Sigma}g\left(\frac{\partial h}{\partial\nu}-h\right)=0.$$
Consequently,

$$Q(f)=Q(h)+Q(g).$$
Since $g|_{\partial\Sigma}\equiv0$, $Q(g)=\int_\Sigma (|\nabla g|^2-|A|^2g^2)$, which is non-negative since $\lambda_1(J)\geq0$, i.e. $\Sigma$ is stable for $J$ with Dirichlet boundary conditions (see Section 3.2). Also, by integration by parts,

$$0=Q(f,v^\perp)=\int_{\partial\Sigma}f\left(\frac{\partial v^\perp}{\partial \nu}-v^\perp\right),$$ and since $v_x^\perp$, $v_y^\perp$ and $v_z^\perp$ are Steklov eigenfunctions with eigenvalue different from $1$, the condition $0=Q(f,v^\perp)$, $v\in\R^3$ is equivalent to

$$0=\int_{\partial\Sigma}fv^\perp,\,v\in\R^3.$$
Thus, 

$$0=\int_{\partial\Sigma}h=\int_{\partial\Sigma}hv^\perp,\,v\in\R^3.$$
According to Proposition \ref{positivity}, one concludes that $Q(h)\geq0$, so $Q(f)\geq0$. The proof of Theorem \ref{CC} is complete.

\cqfd

\section{A general lower bound for the index}

In this section, we prove a general lower bound for the index of a free boundary minimal hypersurface $\Sigma^{n}$ of dimension $n$ in the unit ball $\Bb^{n+1}$ of $\R^{n+1}$, $n\geq2$. In order to do so, we will first need the following lemma, which will also be used in the next section:

\begin{Lem}\label{aux_quad}

Let $\Sigma^n$ be a free boundary minimal hypersurface in $\B^{n+1}$, which is not a flat $\Bb^n$. Let $Q$ be the quadratic form coming from the second variation of $\Sigma^n$, and let $S$ be the auxiliary quadratic form defined by

$$S(u)=\int_{\Sigma^{n}} |\nabla u|^2-\int_{\partial \Sigma^{n}}u^2,\,\,u\in C^\infty(\Sigma).$$
Then, for all $u\in C^\infty(\Sigma^n)\setminus\{0\}$,

$$Q(u)<S(u).$$
Moreover, $S\leq 0$ on the vector space $\mathcal{V}$ spanned by $\mathbf{1}$ and $\{x_v,\,;\,v\in \R^{n+1}\}$, where by definition $x_v:=(x,v)$ for $v\in \R^{n+1}$.

\end{Lem}

\begin{proof}
Recall that $Q$ is given by the formula

$$Q(u)=\int_{\Sigma^{n}} |\nabla u|^2-|A|^2u^2-\int_{\partial \Sigma^{n}}u^2,\,\,u\in C^\infty(\Sigma).$$
Therefore, it is obvious that for $u\in C^\infty(\Sigma^n)$,

$$Q(u)\leq S(u),$$
the inequality being strict unless $|A||u|$ vanishes identically. We argue by contradiction: let $u\in C^\infty(\Sigma^n)\setminus\{0\}$, and assume that $|A||u|$ vanishes identically. By definition, $|A|$ vanishes at a point $p$, if and only if the Gauss map $N:\Sigma^n\to S^n$ has a critical point at $p$. Denote by $E\subset \Sigma^n$ the set of critical points of the Gauss map. The fact that $|A||u|$ vanishes identically implies that $u$ vanishes identically on $\Sigma^n\setminus E$. Since $u$ is continuous and not identically zero, $\Sigma^n\setminus E$ cannot be dense in $\Sigma^n$, or equivalently, $E$ must have an interior point $p$ in $\Sigma^n$. This implies that the Gauss map is constant in a neighborhood of $p$, hence that a small piece of $\Sigma^n$ is included in a ball $\Bb^n$; by an argument involving the unique continuation principle for the minimal graph equation, $\Sigma^n$ must be a ball $\Bb^n$, which is a contradiction. Thus, we have proved that for every $u\in C^\infty(\Sigma^n)\setminus \{0\}$, the function $|A||u|$ cannot vanish identically. As we have seen, this implies that $Q(u)<S(u)$, and the proof of the first part of the lemma is complete.

Concerning the second part, we argue as follows. By \cite{FS1}, since $\Sigma^{n}$ is a free boundary minimal hypersurface in $\Bb^{n+1}$, the functions $x_v$ are Steklov eigenfunctions with eigenvalue $1$. That is, $\Delta_{\Sigma^n}x_v=0$ in the interior of $\Sigma$, and 

$$\frac{\partial x_v}{\partial \nu}=x_v\mbox{ on }\partial\Sigma.$$
The variational characterization of Steklov eigenvalues (see \cite[p.4014]{FS1}) implies that

$$\int_{\partial\Sigma^{n}}x_v=0$$
(this can also be checked directly using Green's formula with $x_v$ and $\mathbf{1}$). Finally, one checks that $S(\mathbf{1})<0$, $S(x_v)=0$, and by integration by parts,

$$S(\mathbf{1},x_v)=\int_{\partial \Sigma^{n}}\left(\frac{\partial x_v}{\partial \nu}-x_v\right)=0.$$
Thus, $S\leq 0$ on $\mathcal{V}$, and this concludes the proof of the second part of the lemma.

\end{proof}
We are now ready to prove a general lower bound for the index.

\begin{Pro}\label{lower_index}

Let $\Sigma^n$ be a free boundary minimal hypersurface in $\B^{n+1}$, which is not a flat $\Bb^n$. Then, the index of $\Sigma$ is at least $n+2$. 

\end{Pro}

\begin{Rem}
{\em 
This improves upon \cite[Theorem 3.1]{FS2}, which implies that (under slightly stronger assumptions than Proposition \ref{lower_index}) the index of $\Sigma^n$ is at least $n+1$.
}
\end{Rem}

\begin{proof}

Recall the auxiliary quadratic form $S$ from Lemma \ref{aux_quad}, defined as

$$S(u)=\int_{\Sigma^{n}} |\nabla u|^2-\int_{\partial \Sigma^{n}}u^2.$$
By the second part of Lemma \ref{aux_quad}, $S\leq 0$ on the vector space $\mathcal{V}$ spanned by $\mathbf{1}$ and $\{x_v\,;\,v\in \R^{n+1}\}$, where $x_v:=(x,v)$. The first part of Lemma \ref{aux_quad} then implies that $Q$ is negative definite on $\mathcal{V}$. Let us show that $\mathcal{V}$ is of dimension $n+2$. Otherwise, there is $a\neq0$ and $v\in\R^{n+1}$ such that $a+x_v=0$ on $\Sigma^{n}$. Then, $\Sigma^{n}$ is included in the hyperplane $\mathcal{H}$ defined by

$$\mathcal{H}=\{x\in \R^{n+1}\,;\,\langle x,v\rangle=-a\}.$$
This means that $\Sigma^{n}$ is included in the flat $n$-ball $\Bb^{n+1}\cap \mathcal{H}$, which necessarily passes through the origin by the free boundary condition, i.e. $a=0$. This contradicts the assumption on $\Sigma^n$. Consequently, we have proved that $\mathcal{V}$ has dimension $n+2$, and the index of $\Sigma^{n}$ is at least $n+2$.

\end{proof}

\section{Free boundary minimal surfaces of index 4}
To conclude this article, we present two results pertaining to Conjecture \ref{uniqueness}. We note that these results have been obtained independently by H. Tran \cite{T}. We start with the following Lemma, which will be employed in the proof of these two results.

\begin{Lem}\label{Steklov4}

If $\Sigma$ is a free boundary minimal surface in $\Bb^3$ with index $4$, then $\sigma_1(\Sigma)$, the first Steklov eigenvalue for $\Delta$, is equal to $1$.

\end{Lem}

\begin{proof}

Assume by contradiction that $\sigma_1(\Sigma)$ is strictly less than $1$, then there is a first eigenfunction $f$ satisfying $\Delta f=0$, $\frac{\partial f}{\partial \nu}=\sigma_1 f$ and 

$$\int_{\partial \Sigma}f=0,\,\int_{\partial\Sigma}fx_v=0,\,v\in \R^3.$$
Let $\mathcal{V}$ be the vector space spanned by $\mathbf{1}$ and $\{x_v\,;\,v\in \R^n\}$, and denote by $\mathcal{W}$ the vector space generated by $\mathcal{V}$ and $f$; it is $5$-dimensional, since $f$ is $L^2$-orthogonal to $\mathcal{V}$ in restriction to $\partial \Sigma$. Recall the auxiliary quadratic form $S$ from Lemma \ref{aux_quad}, defined by

$$S(u)=\int_{\Sigma} |\nabla u|^2-\int_{\partial \Sigma}u^2.$$
It follows from Green's formula that for every $g\in \mathcal{V}$,

$$S(f,g)=(\sigma_1-1)\int_{\partial \Sigma}fg=0.$$
Moreover,

$S(f)=(\sigma_1-1)\int_{\partial \Sigma}f^2<0$
A flat disk passing through the origin having index $1\neq 4$, $\Sigma$ is not a flat disk. Lemma \ref{aux_quad} implies that $Q$ is strictly negative on $\mathcal{W}$, which contradicts the fact that $Q$ has index $4$. Therefore, $\sigma_1(\Sigma)=1$.

\end{proof}

\begin{Cor}\label{stable}

Let $\Sigma$ be an oriented, free boundary minimal surface in $\Bb^3$ with index $4$. Then, the Jacobi operator on $\Sigma$ with Dirichlet boundary conditions has first eigenvalue equal to zero, and in particular $\Sigma$ is stable for perturbations fixing its boundary. More precisely, the normal to $\Sigma$ can be chosen so that the function $\xi:=(X,N)$ is a positive Jacobi field in the interior of $\Sigma$. If moreover $\partial\Sigma$ is embedded, then $\Sigma$ has genus zero.

\end{Cor}

\begin{proof}

The stability of $\Sigma$ follows from the proof of \cite[Prop. 8.1]{FS2}; for the sake of completeness, let us recall the argument. By the well-known Fischer-Colbrie-Schoen/Allegretto-Moss-Piepenbrink lemma (see \cite{FCS}, \cite{MP}), since $(X,N)$ is a Jacobi field, it is enough to show that it does not vanish on $\Sigma$ to conclude that it is a first eigenfunction for $J$ with Dirichlet boundary conditions. By Lemma \eqref{Steklov4}, we know that for every $v\in \R^3\setminus\{0\}$, $x_v:=(X,v)$ is a first Steklov eigenfucntion of $\Delta$. By the nodal theorem for Steklov eigenfunctions, $x_v$ has at most $2$ nodal domains. According to S.Y. Cheng \cite{C}, at a point $p$ belonging to the zero set of $x_v$, $N$ nodal half-lines meet at $p$ if and only if $x_v$ vanishes at $p$ to order $N$. Since $x_v$ has at most $2$ nodal domains, it follows that it vanishes only at order one on its zero set. Since this is true for any $v\in \R^3\setminus \{0\}$, it follows that any plane passing through the origin intersects $\Sigma$ transversally. Thus, the affine tangent plane to an interior point cannot pass through zero. Assume now that $\xi$ vanishes at an interior point $p$. Thus, $X$ belongs to the tangent plane at $p$. But then, the line containing $p$ and having direction $X$ passes through the origin, and is in the affine tangent plane at $p$, which is a contradiction. Therefore, $\xi$ does not vanish in the interior, and we can choose the normal so that it is positive everywhere.\\

Assume now that $\partial\Sigma$ is embedded. Define a surface $\tilde{\Sigma}$ by gluing a disk on each connected component of $\Sigma$ and smoothing out the corners of the obtained surface. Then, the map $p\to \frac{p}{|p|}$ is a local homeomorphism from $\tilde{\Sigma}$ into $S^2$. Thus, it is a covering map, and since $S^2$ is simply connected, it must be a global homeomorphism. So, $\Sigma$ has genus zero.

\end{proof}
An immediate consequence of the Lemma \eqref{Steklov4}, together with \cite[Theorem 6.6]{FS2}, is the following result:

\begin{Cor}

Let $\Sigma$ be an oriented, free boundary minimal surface in $\Bb^3$. Assume that $\Sigma$ has index $4$, and is topologically an annulus. Then, $\Sigma$ is congruent to the critical catenoid.

\end{Cor}

\bibliographystyle{alpha}

\end{document}